\documentclass{article}

\usepackage{amsmath,amssymb,latexsym,graphics,epsfig}
\usepackage{color}
\usepackage{amsthm}
\usepackage{graphicx,url}

\def\0{{\bf 0}}
\def\1{{\bf 1}}
\def\v{{\bf v}}

\newtheorem{thm}{Theorem}[section]
\newtheorem{prp}[thm]{Proposition}
\newtheorem{lem}[thm]{Lemma}

\newtheorem{cor}[thm]{Corollary}
\newtheorem{example}[thm]{Example}

\def\A{\mbox{\boldmath $A$}}
\def\B{\mbox{\boldmath $B$}} 
\def\C{\mbox{\boldmath $C$}}

\def\C{\mathbb C}

\def\diag{\mathop{\rm diag }\nolimits}

\def\v{\mbox{\boldmath $v$}}

\def\O{\mbox{\boldmath $O$}}

\def\A{\mbox{\boldmath $A$}}
\def\B{\mbox{\boldmath $B$}}
\def\C{\mbox{\boldmath $C$}}

\def\D{\mbox{\boldmath $D$}}

\def\I{\mbox{\boldmath $I$}}
\def\J{\mbox{\boldmath $J$}}
\def\L{\mbox{\boldmath $L$}}
\def\M{\mbox{\boldmath $M$}}

\def\Q{\mbox{\boldmath $Q$}}

\def\S{\mbox{\boldmath $S$}}

\def\diag{\mathop{\rm diag }\nolimits}

\def\tr{\mathop{\rm tr }\nolimits}

\begin{document}

\title{An Interlacing Approach for Bounding \\
the Sum of Laplacian Eigenvalues of Graphs}

\author{A. Abiad$^a$, M.A. Fiol$^b$, W.H. Haemers$^a$, G. Perarnau$^b$
\\ \\
{\small $^a$Tilburg University, Dept. of Econometrics and O.R.} \\
{\small  Tilburg, The Netherlands}\\
{\small (e-mails: {\tt
\{A.AbiadMonge,Haemers\}@uvt.nl)}} \\
{\small $^b$Universitat Polit\`ecnica de Catalunya, BarcelonaTech} \\
{\small Dept. de Matem\`atica Aplicada IV, Barcelona, Catalonia}\\
{\small (e-mails: {\tt
\{fiol,guillem.perarnau\}@ma4.upc.edu})} \\
 }
\maketitle
\begin{abstract}
We apply eigenvalue interlacing techniques for obtaining lower
and upper bounds for the sums of Laplacian eigenvalues of
graphs, and characterize equality. This leads to
generalizations of, and variations on theorems by Grone, and
Grone \& Merris. As a consequence we obtain inequalities
involving bounds for some well-known parameters of a graph,
such as edge-connectivity, and the isoperimetric number.
\end{abstract}

\section{Eigenvalue interlacing}
Throughout this paper, $G=(V,E)$ is a finite simple graph with
$n=|V|$ vertices.
Recall that the Laplacian matrix of $G$ is $\L=\D-\A$ where $\D$ is the diagonal matrix
of the vertex degrees and $\A$ is the adjacency matrix of $G$.
Let us also recall the following basic result about interlacing
(see~\cite{H1995}, \cite{F1999}, or~\cite{BH2012}).
\begin{thm}\label{interlacing}
Let $\A$ be a real symmetric $n\times n$  matrix with
eigenvalues $\lambda_1\ge\cdots\ge \lambda_n$. For some $m<n$,
let $\S$ be a real $n\times m$ matrix with orthonormal columns,
$\S^{\top}\S=\I$, and consider the matrix $\B=\S^{\top}\A\S$,
with eigenvalues $\mu_1\ge\cdots\ge \mu_m$. Then,
\begin{itemize}
\item[$(a)$] the eigenvalues of $\B$ interlace those of
    $\A$, that is,
\begin{equation}
\label{ineq:interlacing}
\lambda_i\ge \mu_i\ge \lambda_{n-m+i},\qquad i=1,\ldots, m,
\end{equation}
\item[$(b)$] if the interlacing is tight, that is, for some
    $0\le k\le m$, $\lambda_i=\mu_i$, $i=1,\ldots,k$, and
    $\mu_i=\lambda_{n-m+i}$, $i=k+1,\ldots, m$, then
    $\S\B=\A\S$.
\end{itemize}
\end{thm}
Two interesting particular cases are obtained by choosing
appropriately the matrix $\S$.
If $\S=[\, \I\ \ \O\,]^\top$, then $\B$ is a principal
submatrix of $\A$.
If ${\cal P}=\{U_1,\ldots,U_m\}$ is a partition of $\{1,\ldots,n\}$
we can take for $\B$ the so-called {\em quotient matrix} of $\A$ with
respect to ${\cal P}$.

The first case gives useful conditions for an induced subgraph
$G'$ of a graph $G$, because the adjacency matrix of $G'$ is a
principal submatrix of the adjacency matrix of $G$. However,
the Laplacian matrix $\L'$ of $G'$ is in general not a
principal submatrix of the Laplacian matrix $\L$ of $G$. But
$\L'+\D'$ is a principal submatrix of $\L$ for some nonnegative
diagonal matrix $\D'$. Therefore the left hand inequalities in
(\ref{ineq:interlacing}) still hold for the Laplacian
eigenvalues, because adding the positive semi-definite matrix
$\D'$ decreases no eigenvalue.

In the case that $\B$ is a quotient matrix of $\A$ with respect
to $\cal P$, the element $b_{ij}$ of $\B$ is the average row
sum of the block $\A_{i,j}$ of $\A$ with rows and columns
indexed by $U_i$ and $U_j$, respectively. If $\cal P$ has
characteristic matrix $\C$ (that is, the columns of $\C$ are
the characteristic vectors of $U_1,\ldots,U_m$) then we take
$\S=\C\D^{-1/2}$, where
$\D=\diag(|U_1|,\ldots,|U_m|)=\C^\top\C$. In this case, the
quotient matrix $\B$ is in general not equal to $\S^\top\A\S$,
but $\B = \D^{-1/2}\S^\top\A\S\D^{1/2}$, and thus $\B$ is
similar to (and therefore has the same spectrum as)
$\S^\top\A\S$. If the interlacing is tight, then $(b)$ of
Theorem~\ref{interlacing} reflects that ${\cal P}$ is an {\em
equitable} (or {\em regular}) partition of $\A$, that is, each
block of the partition has constant row and column sums. In
case $\A$ is the adjacency matrix of a graph $G$, equitability
of $\cal P$ implies that the bipartite induced subgraph
$G[U_i,U_j]$ is biregular for each $i\neq j$, and that the
induced subgraph $G[U_i]$ is regular for each
$i\in\{1,\ldots,m\}$. In case of tight interlacing for the
quotient matrix of the Laplacian matrix of $G$, the first
condition also holds, but the induced subgraphs $G[U_i]$ are
not necessarily regular (in this case we speak about an {\em
almost equitable}, or {\em almost regular} partition of $G$).

If a symmetric matrix $\A$ has an equitable partition,
we have the following well-known and useful result (\cite{BH2012}, Section~2.3).

\begin{lem}\label{lem:equitable}
Let $\A$ be a symmetric matrix of order $n$, and suppose ${\cal P}$
is a partition of $\{1,\ldots,n\}$ such that the corresponding
partition of $\A$ is equitable with quotient matrix $\B$.
Then the spectrum of $\B$ is a sub(multi)set of the spectrum of $\A$,
and all corresponding eigenvectors of $\A$ are in the column space of
the characteristic matrix $\C$ of $\cal P$
(this means that the entries of the eigenvector are constant on each partition class $U_i$).
The remaining eigenvectors of $\A$ are orthogonal to the columns of $\C$
and the corresponding eigenvalues remain unchanged if the blocks $\A_{i,j}$
are replaced by $\A_{i,j}+c_{i,j}\J$ for certain constants $c_{i,j}$
(as usual, $\J$ is the all-one matrix).
\end{lem}

Assuming that $G$ has $n$ vertices, with degrees $d_{1}\geq
d_{2}\geq \cdots \geq d_{n}$, and Laplacian matrix $\L$ with
eigenvalues $\lambda_{1}\geq \lambda_{2}\geq \cdots \geq
\lambda_{n}(=0)$, it is known that, for $1\le m\le n$,
\begin{equation}
\label{eq:1HaemersSeminar}
\sum_{i=1}^{m}\lambda_{i}\geq
\sum_{i=1}^{m}d_{i}.
\end{equation}
This is a consequence of Schur's theorem~\cite{S1923} stating
that the spectrum of any symmetric, positive definite matrix
majorizes its main diagonal. In particular, note that if $m=n$
we have equality in \eqref{eq:1HaemersSeminar}, because both
terms correspond to the trace of $\L$.
To prove~\eqref{eq:1HaemersSeminar} by using interlacing, let $\B$
be a principal $m\times m$ submatrix of $\L$ indexed by the
subindexes corresponding to the $m$ higher degrees, with
eigenvalues $\mu_{1}\geq \mu_{2}\geq \cdots \geq \mu_{m}$.
Then,
\begin{equation*}
\tr \B= \sum_{i=1}^{m}d_{i}=
\sum_{i=1}^{m}\mu_{i},
\end{equation*}
and, by interlacing,  $\lambda_{n-m+i}\leq \mu_i\le \lambda_i$
for $i=1,\ldots,m$, whence \eqref{eq:1HaemersSeminar} follows.
Similarly, reasoning with  the principal submatrix $\B$ (of
$\L$) indexed by the $m$ vertices with lower degrees we get:
\begin{equation}
\label{eq:1HaemersSeminar+}
\sum_{i=1}^{m}\lambda_{n-m+i}\leq
\sum_{i=1}^{m}d_{n-m+i}.
\end{equation}

The next result, which is an improvement
of~\eqref{eq:1HaemersSeminar}, is due to Grone~\cite{G1995},
who proved that if $G$ is connected and $m<n$ then,
\begin{equation}
\label{eq:Gronebound}
\sum_{i=1}^{m}\lambda_{i}\geq
\sum_{i=1}^{m}d_{i}+1.
\end{equation}

In \cite{BH2012}, Brouwer and Haemers gave two different proofs
of (\ref{eq:Gronebound}), both using eigenvalue interlacing.
In this paper we extend the ideas of these two proofs
and find a generalization of Grone's result (\ref{eq:Gronebound}), and
another lower bound on the sum of the largest Laplacian eigenvalues,
which is closely related to a bound of Grone and Merris~\cite{GM1994}.

\section{A generalization of Grone's result}
We begin with a basic result from where most of our bounds
derive. Given a graph $G$ with a vertex subset $U\subset V$,
let $\partial U$ be the \emph{vertex-boundary} of $U$, that is,
the set of vertices in $\overline{U}=V\backslash U$ with at least
one adjacent vertex in  $U$.
Also, let $\partial(U,\overline{U})$ denote the \emph{edge-boundary} of
$U$, which is the set of edges which connect vertices in $U$
with vertices in~$\overline{U}$.

\begin{thm}
\label{thm:basic-result} Let $G$ be a connected graph on
$n=|V|$ vertices, having Laplacian matrix $\L$ with eigenvalues
$\lambda_{1}\geq \lambda_{2}\geq \cdots \geq \lambda_{n}(=0)$.
For any given vertex subset $U=\{u_1,\ldots,u_m\}$ with $0< m<
n$, we have
\begin{equation}
\label{basic-ineq}
\sum_{i=1}^{m}\lambda_{n-i}\le \sum_{u\in U}d_{u} + \frac{|\partial(U,\overline{U})|}{n-m} \le  \sum_{i=1}^{m}\lambda_{i}.
\end{equation}
\end{thm}
\begin{proof}
 Consider the partition
 of the vertex set $V$ into $m+1$ parts: $U_i=\{u_i\}$ for
 $u_i\in U$, $i=1,\ldots,m$, and $U_{m+1}=\overline{U}$.
 Then, the corresponding quotient matrix is
 $$
 \B=\left[\begin{array}{ccc|c}
  &  &  & b_{1,m+1}\\
  & \L_{U} &  & \vdots\\
  & &  & b_{m,m+1}\\ \hline b_{m+1,1} & \cdots &
 b_{m+1,m} & b_{m+1,m+1}
 \end{array} \right],
 $$
 where $\L_{U}$ is the principal submatrix of $\L$, with rows and columns indexed by the vertices in $U$,
 $b_{i,m+1}=(n-m)b_{m+1,i}=-|\partial(U_i,\overline{U})|$, and
 $b_{m+1,m+1}=|\partial(U,\overline{U})|/(n-m)$ (because $\sum_{i=1}^{m+1}b_{m+1,i}=0$).
 Let $\mu_1\ge \mu_2\ge \cdots\ge \mu_{m+1}$ be the eigenvalues of $\B$.
 Since $\B$ has row sum $0$, we have $\mu_{m+1}=\lambda_n=0$.
 Moreover,
 \begin{equation*}
 \sum_{i=1}^{m}\mu_{i}= \sum_{i=1}^{m+1}\mu_{i}= \tr \B = \sum_{u\in U}d_{u}+b_{m+1,m+1},
 \end{equation*}
 Then, (\ref{basic-ineq}) follows by applying interlacing, $\lambda_{i}\geq \mu_i \geq \lambda_{n-m-1+i}$
 and adding up for $i=1,2,\ldots,m$.
\end{proof}
If equality holds on either side of~(\ref{basic-ineq}) it follows that
the interlacing is tight (see the proof of Proposition~\ref{equality} for details),
and therefore that the partition of $G$ is almost equitable.
In other words, in case of equality every vertex $u \in U$ is adjacent to either all 
or $0$ vertices in  $\overline{U}$, whereas each vertex
$u\in\overline{U}$ has precisely
$|\partial(U,\overline{U})|/(n-m)$ neighbors in $U$.
But we can be more precise.
\begin{prp}\label{equality}
Let $H$ be the subgraph of $G$ induced by $\overline{U}$,
and let $\vartheta_1\geq\cdots\geq\vartheta_{n-m}(=0)$ be the Laplacian eigenvalues of $H$.
Define $b=|\partial(U,\overline{U})|/(n-m)$.
\begin{itemize}\item[$(a)$] 
Equality holds on the right hand side of
$(\ref{basic-ineq})$ if and only if each vertex of $U$ is
adjacent to all or $0$ vertices of $\overline{U}$, and
$\lambda_{m+1}=\vartheta_{1}+b$.
\item[$(b)$] 
Equality holds on the left hand side of
$(\ref{basic-ineq})$ if and only if each vertex of $U$ is
adjacent to all or $0$ vertices of $\overline{U}$, and
$\lambda_{n-m-1}=\vartheta_{n-m}+b$.
\end{itemize}
\end{prp}
\begin{proof}
Since $(a)$ and $(b)$ have analogous proofs, we only prove
$(a)$. Suppose equality holds on the right hand side of
(\ref{basic-ineq}). Then
\[
\sum_{i=1}^m \lambda_i = \sum_{i=1}^m\mu_i, \, \mbox{ and }\, \lambda_i\geq\mu_i \,\mbox{ for }\, i=1,\ldots,m
\]
so $\lambda_i=\mu_i$ for $i=1,\ldots,m$. We know that $\mu_{m+1}=\lambda_n=0$, therefore the interlacing is tight
and hence the partition of $G$ is almost equitable.
Now by use of Lemma~\ref{lem:equitable} we have that the eigenvalues of $\L$ are $\mu_1,\ldots,\mu_{m+1}$ together with
the eigenvalues of $\L$ with an eigenvector orthogonal to the characteristic matrix $\C$ of the partition.
These eigenvalues and eigenvectors remain unchanged if $\L$ is changed into
\[ \widetilde{\L}=\left[\begin{array}{cc} \O & \O\\ \O & \L_{\overline{U}}+b\I \end{array}\right].
\]
The considered eigenvalues of $\widetilde{\L}$ and $\L$ 
are $\vartheta_1 + b \geq \cdots \geq \vartheta_{n-m-1}+b$. So
$\L$ has eigenvalues $\lambda_1(=\mu_1)\ge \cdots \ge
\lambda_m(=\mu_m),\mbox{ and }  \vartheta_1+b \geq \cdots \geq
\vartheta_{n-m-1}+b \ge \lambda_n(=\mu_{m+1}=0)$. Hence, we have
$\lambda_{m+1}=\vartheta_{1}+b$. Conversely, if the
partition of $G$ is almost equitable (or equivalently, if the
partition of $\L$ is equitable), $\L$ has eigenvalues
$\mu_1\geq\cdots\geq\mu_m$, $\vartheta_1
+b\geq\cdots\geq\vartheta_{n-m-1}+b$, and
$\mu_{m+1}=\lambda_{n}=0$. Since
$\lambda_{m+1}=\vartheta_{1}+b$, if follows that
$\mu_i=\lambda_i$ for $i=1,\ldots,m$ (tight interlacing),
therefore equality holds on the right hand side of
(\ref{basic-ineq}).
\end{proof}

Looking for examples of the above results, first  observe that there is no graph with $n>2$
satisfying equality in \eqref{eq:Gronebound} for every $0<m<n$.
However the complete graph $K_n$ provides an example for which both inequalities
in Theorem~\ref{thm:basic-result} are
equalities for all $0<m<n$. 
In fact, this is a particular case of the following contruction
(just take $q=1$): Let us consider the {\em graph join} $G$ of
the complete graph $K_p$ with the empty graph $\overline{K_q}$.
(Recall that $G$ is obtained as the graph union of $K_p$ and
$\overline{K_q}$ with all the edges connecting the vertices of
one graph with the vertices of the other.) Let
$V(G)=\{v_1,\ldots,v_p,v_{p+1},\ldots,v_{n}\}$, where $n=p+q$
and the first vertices correspond to those of $K_p$. Then, the
Laplacian eigenvalues of $G$ are $\{n^{p},p^{q-1},0^{1}\}$, and
the following different choices for $U$ provide
some examples illustrating cases $(a)$ and $(b)$ of Proposition \ref{equality}.
\begin{itemize}
\item[$(a1)$]
Let $U=\{v_1,\ldots,v_m\}$, with $0<m\le p$. Then, $b=m$, and
$$
\sum_{u\in U}d_{u} + b = m(n-1)+m= mn=
\sum_{i=1}^m \lambda_i.
$$
\item[$(a2)$]
Let $U=\{v_1,\ldots,v_m\}$, with $p<m<n$. Then, $b=p$, and
$$
\sum_{u\in U}d_{u} + b = p(n-1)+(m-p)p+p=pn+(m-p)p=
\sum_{i=1}^m \lambda_i.
$$
\item[$(b)$]
Let $U=\{v_{n-m+1},\ldots,v_n\}$, with $q\le m<n$. Then, $b=m$, and
$$
\sum_{u\in U}d_{u} + b = qp+(m-q)(n-1)+m=(q-1)p+(m-q+1)n=
\sum_{i=1}^m \lambda_{n-i}.
$$
\end{itemize}
%
Another infinite family of graphs for which we do have equality
on the right hand side of (\ref{basic-ineq}) is the complete
multipartite graph (such that the vertices with largest degree
lie in $U$).

If the vertex degrees of $G$ are $d_1\ge d_2\ge \cdots \ge
d_n$, we can choose conveniently the $m$ vertices of $U$ (that
is, those with higher or lower degrees) to obtain the best
inequalities in~\eqref{basic-ineq}. Namely,
\begin{equation}
\label{eq:Aida1}
\sum_{i=1}^{m}\lambda_{i} \geq \sum_{i=1}^m d_{i} + \frac{|\partial(U,\overline{U})|}{n-m},
\end{equation}
and
\begin{equation}
\label{eq:Aida1'}
\sum_{i=1}^{m}\lambda_{n-i}\le \sum_{i=1}^m d_{n-i+1} + \frac{|\partial(U,\overline{U})|}{n-m}.
\end{equation}
Note that
\eqref{eq:Aida1'}, together with~\eqref{eq:1HaemersSeminar+} for $m+1$, yields
\begin{equation}
\label{eq:Aida1''}
\sum_{i=0}^{m}\lambda_{n-m+i}=\sum_{i=1}^{m}\lambda_{n-i}\le \sum_{i=1}^m
d_{n-i+1}+\min\left\{d_{n-m},\frac{|\partial(U,\overline{U})|}{n-m}\right\}.
\end{equation}
%
%
If we have more information on the structure of the graph, we
can improve the above results by either bounding
$|\partial(U,\overline{U})|$ or `optimizing'  the ratio
$b=|\partial(U,\overline{U})|/(n-m)$. In fact, the right
inequality in~\eqref{basic-ineq} (and, hence,~\eqref{eq:Aida1})
can be improved when $\overline{U}\neq\partial U$. Simply first
delete the vertices (and corresponding edges) of
$\overline{U}\setminus\partial U$, and then apply the
inequality. Then $d_1,\ldots,d_m$ remain the same and
$\lambda_1,\ldots,\lambda_m$ do not increase. Thus we obtain:
\begin{thm}
\label{propo:gen-Grone} Let $G$ be a connected graph on $n=|V|$
vertices, with Laplacian eigenvalues $\lambda_{1}\geq
\lambda_{2}\geq\cdots \geq \lambda_{n}(=0)$. For any given
vertex subset $U=\{u_1,\ldots,u_m\}$ with $0< m< n$,
we have
\begin{equation}
\label{eq:gen-Grone}
\sum_{i=1}^{m}\lambda_{i} \geq \sum_{u\in U} d_{u} + \frac{|\partial(U,\overline{U})|}{|\partial U|}.
\end{equation}
\end{thm}
Similarly as we did in~\eqref{eq:Aida1}, if we choose the $m$
vertices of $U$ such that they are those with maximum degree,
then we can write:
$$
\sum_{i=1}^{m}\lambda_{i} \geq \sum_{i=1}^m d_i + \frac{|\partial(U,\overline{U})|}{|\partial U|}.
$$
Notice that, as a corollary, we get Grone's result~(\ref{eq:Gronebound})
since always $|\partial(U,\overline{U})|\geq|\partial U|$.

\section{A variation of a bound by Grone and Merris}
%
%
In~\cite{GM1994}, Grone and Merris gave another lower bound for
the sum of the Laplacian eigenvalues, in the case when there is
an induced subgraph consisting of isolated vertices and edges.
Let $G$ be a connected graph of order $n> 2$ with Laplacian
eigenvalues $\lambda_1\ge \lambda_2\ge\cdots\ge\lambda_n$. If
the induced subgraph of a subset $U\subset V$ with $|U|=m$
consists of $r$ pairwise disjoint edges and $m-2r$ isolated
vertices, then
\begin{equation}
\label{eq:GroneMerris}
 \sum_{i=1}^{m}\lambda_i \geq  \sum_{u\in U}d_u +m-r.
\end{equation}


An improvement of this result was given by Brouwer and Haemers
in~\cite{BH2012} (Section 3.10). Let $G$ be a (not necessarily
connected) graph with a vertex subset $U$, with $m=|U|$, and
let $h$ be the number of connected components of $G[U]$ that
are not connected components of $G$. Then,
\begin{equation}
\label{eq:BH'book}
\sum_{i=1}^{m}\lambda_i \geq \sum_{u\in U}d_u +h.
\end{equation}

Following the same ideas as in~\cite{GM1994} and using
interlacing, the bound~\eqref{eq:GroneMerris} of Grone and
Merris can also be generalized as follows:

 \begin{thm}
 \label{thm:gen-GroneMerris}
 Let $G$ be a connected graph of order $n>2$ with Laplacian eigenvalues $\lambda_1\ge \lambda_2\ge\cdots\ge\lambda_n$.
 Given a vertex subset $U\subset V$, with $m=|U|<n$, let $G[U]=(U,E[U])$ be its induced subgraph. Then,
 \begin{equation}
 \label{eq:gen-GroneMerris}
 \sum_{i=1}^{m}\lambda_i \geq  \sum_{u\in U}d_u +m-|E[U]|.
 \end{equation}
 \end{thm}

 \begin{proof}
 Consider an orientation of $G$ with all edges in $E(U,\overline{U})$ oriented from $U$ to $\overline{U}$, and every vertex in $U\setminus\partial{\overline{U}}$ having some outgoing arc (this is always possible as $G$ is connected).
 Let $\Q$ be the corresponding oriented incidence matrix of $G$, and write
 $\Q=[\, \Q_1\ \Q_2\,]$, where $\Q_1$ corresponds to $E[U] \cup E(U,\overline{U})$, and $\Q_2$ corresponds to $E[\overline{U}]$.
 Consider the matrix $\M=\Q^{\top}\Q$, with entries $(\M)_{ii}=2$,
 $(\M)_{ij}=\pm 1$ if the arcs $e_i,e_j$ are incident to the same vertex ($+1$ if
 both are either outgoing or ingoing, and $-1$ otherwise), and $(\M)_{ij}=0$ if
 the corresponding edges are disjoint, and define $\M_1=\Q_1^\top\Q_1$.
 Then $\M$ has the same nonzero eigenvalues as $\L=\Q \Q^\top$, the Laplacian matrix of $G$, and
 $\M_1$ is a principal submatrix of $\M$.
 For every vertex $u\in U$, let $E_u$ be the set of outgoing arcs from $u$.
 Then $\{E_u\, | \, u\in U\}$ is a partition of $E[U] \cup E(U,\overline{U})$.
 Consider the quotient matrix $\B_1 = (b_{ij})$ of $\M_1$ with respect to this partition.
 Then, $b_{uu}= d^{+}(u)+1$ for each $u\in U$.
 Let $\mu_1\ge \mu_2\ge \cdots\ge \mu_{m}$ be the eigenvalues of $\B_1$, then
 \begin{equation*}
 \tr \B_1 = \sum_{i=1}^{m}\mu_{i} = \sum_{u\in U}d^+_{u}+m=\sum_{u\in U}d_{u}-|E[U]|+m
 \end{equation*}
 and (\ref{eq:gen-GroneMerris}) follows since the eigenvalues of $\B_1$ interlace those of
 $\M_1$, which in turn interlace those of $\M$.
 \end{proof}

Note that (\ref{eq:gen-GroneMerris}) also follows from Equation~(\ref{eq:BH'book}).
However, the result can be improved by considering the partition
${\cal P} = \{E_u\, | u\in U\}\cup \{E[\overline{U}]\}$ of the whole edge set of $G$.

\begin{thm}
\label{thm:gen-GroneMerris2} Let $G$ be a connected graph of
order $n>2$ with Laplacian eigenvalues $\lambda_1\ge \lambda_2\ge\cdots\ge\lambda_n$.
Given a vertex subset $U\subset V$, with $m=|U|<n$, let $G[U]=(U,E[U])$ and
$G[\overline{U}]$ be the corresponding induced subgraphs.
Let $\vartheta_1$ be the largest Laplacian eigenvalue of
$G[\overline{U}]$, then
\begin{equation}
\label{eq:gen-GroneMerris2}
\sum_{i=1}^{m+1}\lambda_i \geq  \sum_{u\in U}d_u +m-|E[U]|
+\vartheta_1.
\end{equation}
\end{thm}

\begin{proof}
First observe that the Laplacian matrix of $G[\overline{U}]$ is
$\Q_2\Q_2^\top$, and therefore $\vartheta_1$ is also the
largest eigenvalue of $\Q_2^\top\Q_2$. Next we apply
interlacing to an $(m+1)\times(m+1)$ quotient matrix
$\B=\S^\top\M\S$, which is defined slightly different as
before. The first $m$ columns of $\S$ are the normalized
characteristic vectors of $E_u$ (as before), but the last
column of $\S$ equals $\left[{\,}^{\0}_{\v}\right]$, where $\v$
is a normalized eigenvector of $\Q_2^\top\Q_2$ for the
eigenvalue $\vartheta_1$. Then
$b_{m+1,m+1}=\v^\top\Q_2^\top\Q_2 \v=\vartheta_1$, and we find
$\tr \B = \sum_{u\in U}d_u+m-|E[U]|+\vartheta_1$.
\end{proof}

\section{Some applications}
The previous bounds on the sum of Laplacian eigenvalues are
used to provide meaningful results involving the edge-connectivity of the graph, the
size of a $k$-dominating set and the isoperimetric number.

\subsection{Cuts}
Given a vertex subset $U$ of a connected graph $G$ with $0<|U|<n$,
the edge set $\partial(U,\overline{U})$ is called a {\em cut}
(since deletion of these edges makes $G$ disconnected).
The minimum size of a cut in $G$ is called the {\em edge-connectivity} $\kappa_e (G)$ of $G$.
By use of inequality~(\ref{eq:Aida1}) we obtain the following bound for $\kappa_e (G)$.
\begin{prp}
\begin{equation}
\label{edge-connec} \kappa_e (G)\le
\min_{0<m<n}\left\{(n-m)\sum_{i=1}^{m}(\lambda_{i}-d_{i})\right\}.
\end{equation}
\end{prp}
Some general bounds on the size of a cut can be derived from the following lemma.
\begin{lem}
\label{lem:Guillem} Let $G$ be a graph with $n$ vertices and $e$
edges. For any m, $0<m< n$, there exist some (not necessarily
different) vertex subsets $U$ and $U'$ such that $|U|=|U'|=m$ and
\begin{equation}
\label{bounds-maxmin-cuts} |\partial(U,\overline{U})|\geq
\frac{2em(n-m)}{n(n-1)}, \qquad |\partial(U',\overline{U'})|\leq
\frac{2em(n-m)}{n(n-1)}.
\end{equation}
\end{lem}
\begin{proof}
Choose a set $S$ uniformly at random among all the sets of size
$m$ in $V$. Then the probability that an edge belong to
$\partial(S,\overline{S})$ is the probability that either the
first endpoint belongs to $S$ and the second one to $\overline{S}$
or viceversa. That is,
$$
\Pr(\text{edge} \in \partial(S,\overline{S})) =
2\frac{m(n-m)}{n(n-1)}.
$$
Then, the expected number of edges between the two sets is,
$$
\mathbb{E}\{|\partial(S,\overline{S})|\}=\frac{2em(n-m)}{n(n-1)},
$$
implying that there are sets, $U$ and $U'$, with at least and at
most this number of edges going out, respectively.
\end{proof}

Both bounds are tight for the complete graph $K_n$.
Using bounds (\ref{bounds-maxmin-cuts}), Theorem~\ref{thm:basic-result} gives:
\begin{cor}\label{cor:1}
For each $m$ ($0<m<n$) $G$ has vertex sets $U$ and $U'$ of size $m$ such that
\begin{equation}
 \sum_{i=1}^{m}\lambda_{i} \geq
\sum_{u\in U} d_{u} + \frac{2em}{n(n-1)},
\end{equation}
and
\begin{equation}
\sum_{i=1}^{m}\lambda_{n-i} \leq \sum_{u\in U'} d_{u} +
\frac{2em}{n(n-1)}.
\end{equation}
\end{cor}
In particular, if $G$ is $d$-regular, we have $e=nd/2$ and the
above inequalities become
\begin{equation}
\label{bounds-reg}
 \sum_{i=1}^{m}\lambda_{i} \geq \frac{mdn}{n-1}\qquad\mbox{and}\qquad \sum_{i=1}^{m}\lambda_{n-i} \leq
\frac{mdn}{n-1},
\end{equation}
with bounds close to $md$ when $n$ is large.
\subsection{$k$-Dominating sets}
\noindent A \emph{dominating set} in a graph $G$ is a vertex
subset $D\subseteq V$ such that every vertex in $V\backslash D$
is adjacent to some vertex in $D$.
More generally, for a given integer $k$,
a \emph{$k$-dominating set} in a graph $G$ is a vertex subset
$D\subseteq V$ such that every vertex not in $D$ has at least
$k$ neighbors in $D$.
\begin{prp}
\label{propo:Aida5} Let $G$ be a graph on $n$ vertices, with
vertex degrees $d_{1}\geq d_{2}\geq \cdots \geq d_{n}$, and
eigenvalues $\lambda_{1}\geq \lambda_{2}\geq \cdots \geq \lambda_{n}(=0)$.
Let $D$ be a $k$-dominating set in $G$ of cardinality $m$.
Then,
\begin{equation}
\label{eq:Aida5}
 \sum_{i=1}^{m}\lambda_{i} \geq
\sum_{u\in D}d_{u} + k.
\end{equation}
\end{prp}
\begin{proof}
First, the inequality \eqref{eq:Aida5} follows from Theorem~\ref{thm:basic-result}
by noting that, from the definition of a $k$-dominating set, $|\partial(D,\overline{D})|\ge k(n-m)$.
\end{proof}
\begin{example} Consider the $K_{p,\ldots,p}$  regular complete multipartite
graph with $q$ classes of size $p$, so $n=pq$ and $d=p(q-1)$.
The eigenvalues of its Laplacian matrix are
\[
\{ (d+p)^{q-1}, d^{n-q}, 0^{1}\}.
\]
Observe that the union of some partition classes gives a
$k$-dominating set of size $m=k$. If we take the first $k$
eigenvalues, the inequality ~\eqref{eq:Aida5} becomes
$(d+p)(q-1)+(k-(q-1))d \geq kd+k$, and using that $d=p(q-1)$ we
get $d(k+1)\geq k(d+1)$. Note that if $k=d$ we have equality.
\end{example}

\subsection{The isoperimetric number}

Given a graph $G$ on $n$ vertices, the \emph{isoperimetric number} $i(G)$ is defined
as
\[
i(G)=\min_{U\subset V} \left\{|\partial(U,\overline{U})|/|U| : 0 < |U|\leq n/2\right\}.
\]
For example, the isoperimetric numbers of the complete
graph, the path and the cycle are, respectively, $i(K_n)=\lceil
\frac{n}{2} \rceil$, $i(P_n)=1/\lfloor \frac{n}{2} \rfloor$,
and  $i(C_n)=2/\lfloor \frac{n}{2} \rfloor$.
For general graphs, Mohar \cite{M1989} proved the following
spectral bounds.
\begin{equation}
\label{eq:Mohar}
\frac{\lambda_{n-1}}{2}\leq i(G) \leq \sqrt{\lambda_{n-1} (2d_1 -\lambda_{n-1})}.
\end{equation}
In our context we have:
\begin{prp}
\label{propo:Aida7b}
\begin{equation}
\label{eq:Aida7b}
i(G)\leq  \min_{\frac{n}{2}\leq m < n} \sum_{i=1}^{m}(\lambda_{i} - d_{i}).
\end{equation}
\end{prp}
\begin{proof}
Apply (\ref{eq:Aida1}) taking into account that
$i(G)\leq\frac{|\partial(\overline{U},U)|}{|\overline{U}|}$ when
$0<|\overline{U}|\leq\frac{n}{2}$.~\end{proof}
\begin{example} Consider the graph join $G$ of
the complete graph $K_p$ with the empty graph $\overline{K_q}$,
so $n=p+q$. The Laplacian spectrum and the degree sequence are
\[
\{ n^{p}, p^{q-1}, 0^{1}\}\ \mbox{ and }\
\{ (n-1)^{p}, p^{q} \},
\]
respectively. Equation~$(\ref{eq:Aida7b})$ gives
$i(G)\leq\min\{p,\lceil\frac{n}{2}\rceil\}$, which is better
than Mohar's upper bound~$(\ref{eq:Mohar})$ for all $0\leq
q<n$.
\end{example}

\subsection*{Acknowledgments}
Research supported by the {\em Ministerio de Ciencia e Innovaci\'on}, Spain, and the {\em European Regional Development Fund} under project MTM2011-28800-C02-01  and by the {\em Catalan Research Council} under project 2009SGR1387 (M.A.F. and G.P.),
and by  {\em The Netherlands Organisation for Scientific Research} $(NWO)$ (A.A.).


\begin{thebibliography}{99}

%
%
%
%
%
\bibitem{BH2012}
A.E. Brouwer and W.H. Haemers,
{\em Spectra of Graphs}, Springer Verlag (Universitext), Heidelberg, 2012;
available online at \url{http://homepages.cwi.nl/~aeb/math/ipm/}.
%
%
\bibitem{F1999}
M.A. Fiol, Eigenvalue interlacing and weight parameters of graphs, \emph{Linear Algebra Appl.} {\bf 290} (1999), 275--301.


\bibitem{G1995}
R. Grone, Eigenvalues and the degree seqences of graphs,
{\em Linear Multilinear Algebra} {\bf 39} (1995) 133--136.


\bibitem{GM1994}
R. Grone and R. Merris,
The Laplacian spectrum of a graph. II,
\emph{SIAM J. Discrete Math.} {\bf 7} (1994), no. 2, 221--229.

\bibitem{H1995}
W.H. Haemers, Interlacing eigenvalues and graphs,  \emph{Linear Algebra Appl.} {\bf 226-228} (1995) 593--616.

%
%

\bibitem{M1989}
B. Mohar, Isoperimetric numbers of graphs, \emph{J. Combin. Theory Ser. B}
{\bf 47} (1989), no. 3,
274--291.
%
%
%
\bibitem{S1923}
I. Schur, \"{U}ber eine Klasse von Mittelbildungen mit Anwendungen die Determinanten, {\em Theorie Sitzungsber.} Berlin. Math. Gessellschaft {\bf 22} (1923) 9--20.

\end{thebibliography}
\end{document}